\documentclass[11pt]{amsart}    

\usepackage{latexsym, amsmath, amsthm, amssymb, setspace,verbatim}
\usepackage[all]{xy}
\usepackage{longtable}
\usepackage[ansinew]{inputenc}
\usepackage{hyperref}
\usepackage{enumerate}
\usepackage{color}

\title[Endomorphisms of the Cuntz Algebras and the Thompson Groups]{Endomorphisms of the Cuntz Algebras and \\ the Thompson Groups}

\author{Sel\c{c}uk Barlak}
\author{Jeong Hee Hong}
\author{Wojciech Szyma{\'n}ski}
\date{\today}
\subjclass{20E36, 46L05, 46L40}
\keywords{Thompson's groups, Cuntz algebra, endomorphism}
\thanks{Jeong Hee Hong was supported by Basic Science Research Program through 
the National Research Foundation of Korea (NRF) funded by the Ministry of Education, Science and 
Technology (Grant No. 2016R1D1A1B03930839). 
Sel\c{c}uk Barlak and Wojciech Szyma\'{n}ski were partially supported by 
the Villum Fonden Research Project `Local and global structures of groups and their algebras' (2014--2018).}

\begin{document}

\maketitle

\begin{abstract}
We investigate the relationship between endomorphisms of the Cuntz algebra ${\mathcal O}_2$ and endomorphisms of the Thompson groups 
$F$, $T$ and $V$ represented inside the unitary group of ${\mathcal O}_2$. For an endomorphism $\lambda_u$ of ${\mathcal O}_2$, we show 
that $\lambda_u(V)\subseteq V$ if and only if $u\in V$. If $\lambda_u$ is an automorphism of ${\mathcal O}_2$ then $u\in V$ is 
equivalent to $\lambda_u(F)\subseteq V$. Our investigations are facilitated by introduction of the concept of modestly scaling endomorphism of 
${\mathcal O}_n$, whose properties and examples are investigated. 
\end{abstract}

\renewcommand\matrix[1]{\left(\begin{array}{*{10}{c}} #1 \end{array}\right)}  
\newcommand\set[1]{\left\{#1\right\}}  
\newcommand\mset[1]{\left\{\!\!\left\{#1\right\}\!\!\right\}}

\def\D{{\mathcal D}}
\def\F{{\mathcal F}}
\def\J{{\mathcal J}}
\newcommand{\Nb}{{\mathbb{N}}}
\def\OO{{\mathcal O}}
\def\SS{{\mathcal S}}
\def\U{{\mathcal U}}
\def\W{{\mathcal W}}

\newcommand{\ff}{\mathfrak f}

\def\Ad{\operatorname{Ad}}
\def\aut{\operatorname{Aut}}
\def\en{\operatorname{End}}

\newcommand{\IA}[0]{\mathbb{A}} \newcommand{\IB}[0]{\mathbb{B}}
\newcommand{\IC}[0]{\mathbb{C}} \newcommand{\ID}[0]{\mathbb{D}}
\newcommand{\IE}[0]{\mathbb{E}} \newcommand{\IF}[0]{\mathbb{F}}
\newcommand{\IG}[0]{\mathbb{G}} \newcommand{\IH}[0]{\mathbb{H}}
\newcommand{\II}[0]{\mathbb{I}} \renewcommand{\IJ}[0]{\mathbb{J}}
\newcommand{\IK}[0]{\mathbb{K}} \newcommand{\IL}[0]{\mathbb{L}}
\newcommand{\IM}[0]{\mathbb{M}} \newcommand{\IN}[0]{\mathbb{N}}
\newcommand{\IO}[0]{\mathbb{O}} \newcommand{\IP}[0]{\mathbb{P}}
\newcommand{\IQ}[0]{\mathbb{Q}} \newcommand{\IR}[0]{\mathbb{R}}
\newcommand{\IS}[0]{\mathbb{S}} \newcommand{\IT}[0]{\mathbb{T}}
\newcommand{\IU}[0]{\mathbb{U}} \newcommand{\IV}[0]{\mathbb{V}}
\newcommand{\IW}[0]{\mathbb{W}} \newcommand{\IX}[0]{\mathbb{X}}
\newcommand{\IY}[0]{\mathbb{Y}} \newcommand{\IZ}[0]{\mathbb{Z}}

\newcommand{\CA}[0]{\mathcal{A}} \newcommand{\CB}[0]{\mathcal{B}}
\newcommand{\CC}[0]{\mathcal{C}} \newcommand{\CD}[0]{\mathcal{D}}
\newcommand{\CE}[0]{\mathcal{E}} \newcommand{\CF}[0]{\mathcal{F}}
\newcommand{\CG}[0]{\mathcal{G}} \newcommand{\CH}[0]{\mathcal{H}}
\newcommand{\CI}[0]{\mathcal{I}} \newcommand{\CJ}[0]{\mathcal{J}}
\newcommand{\CK}[0]{\mathcal{K}} \newcommand{\CL}[0]{\mathcal{L}}
\newcommand{\CM}[0]{\mathcal{M}} \newcommand{\CN}[0]{\mathcal{N}}
\newcommand{\CO}[0]{\mathcal{O}} \newcommand{\CP}[0]{\mathcal{P}}
\newcommand{\CQ}[0]{\mathcal{Q}} \newcommand{\CR}[0]{\mathcal{R}}
\newcommand{\CS}[0]{\mathcal{S}} \newcommand{\CT}[0]{\mathcal{T}}
\newcommand{\CU}[0]{\mathcal{U}} \newcommand{\CV}[0]{\mathcal{V}}
\newcommand{\CW}[0]{\mathcal{W}} \newcommand{\CX}[0]{\mathcal{X}}
\newcommand{\CY}[0]{\mathcal{Y}} \newcommand{\CZ}[0]{\mathcal{Z}}

\newcommand{\FA}[0]{\mathfrak{A}} \newcommand{\FB}[0]{\mathfrak{B}}
\newcommand{\FC}[0]{\mathfrak{C}} \newcommand{\FD}[0]{\mathfrak{D}}
\newcommand{\FE}[0]{\mathfrak{E}} \newcommand{\FF}[0]{\mathfrak{F}}
\newcommand{\FG}[0]{\mathfrak{G}} \newcommand{\FH}[0]{\mathfrak{H}}
\newcommand{\FI}[0]{\mathfrak{I}} \newcommand{\FJ}[0]{\mathfrak{J}}
\newcommand{\FK}[0]{\mathfrak{K}} \newcommand{\FL}[0]{\mathfrak{L}}
\newcommand{\FM}[0]{\mathfrak{M}} \newcommand{\FN}[0]{\mathfrak{N}}
\newcommand{\FO}[0]{\mathfrak{O}} \newcommand{\FP}[0]{\mathfrak{P}}
\newcommand{\FQ}[0]{\mathfrak{Q}} \newcommand{\FR}[0]{\mathfrak{R}}
\newcommand{\FS}[0]{\mathfrak{S}} \newcommand{\FT}[0]{\mathfrak{T}}
\newcommand{\FU}[0]{\mathfrak{U}} \newcommand{\FV}[0]{\mathfrak{V}}
\newcommand{\FW}[0]{\mathfrak{W}} \newcommand{\FX}[0]{\mathfrak{X}}
\newcommand{\FY}[0]{\mathfrak{Y}} \newcommand{\FZ}[0]{\mathfrak{Z}}

\newcommand{\Fa}[0]{\mathfrak{a}} \newcommand{\Fb}[0]{\mathfrak{b}}
\newcommand{\Fc}[0]{\mathfrak{c}} \newcommand{\Fd}[0]{\mathfrak{d}}
\newcommand{\Fe}[0]{\mathfrak{e}} \newcommand{\Ff}[0]{\mathfrak{f}}
\newcommand{\Fg}[0]{\mathfrak{g}} \newcommand{\Fh}[0]{\mathfrak{h}}
\newcommand{\Fi}[0]{\mathfrak{i}} \newcommand{\Fj}[0]{\mathfrak{j}}
\newcommand{\Fk}[0]{\mathfrak{k}} \newcommand{\Fl}[0]{\mathfrak{l}}
\newcommand{\Fm}[0]{\mathfrak{m}} \newcommand{\Fn}[0]{\mathfrak{n}}
\newcommand{\Fo}[0]{\mathfrak{o}} \newcommand{\Fp}[0]{\mathfrak{p}}
\newcommand{\Fq}[0]{\mathfrak{q}} \newcommand{\Fr}[0]{\mathfrak{r}}
\newcommand{\Fs}[0]{\mathfrak{s}} \newcommand{\Ft}[0]{\mathfrak{t}}
\newcommand{\Fu}[0]{\mathfrak{u}} \newcommand{\Fv}[0]{\mathfrak{v}}
\newcommand{\Fw}[0]{\mathfrak{w}} \newcommand{\Fx}[0]{\mathfrak{x}}
\newcommand{\Fy}[0]{\mathfrak{y}} \newcommand{\Fz}[0]{\mathfrak{z}}

\newcommand{\Ra}[0]{\Rightarrow}
\newcommand{\La}[0]{\Leftarrow}
\newcommand{\LRa}[0]{\Leftrightarrow}

\renewcommand{\phi}[0]{\varphi}
\newcommand{\eps}[0]{\varepsilon}

\newcommand{\quer}[0]{\overline}
\newcommand{\uber}[0]{\choose}
\newcommand{\ord}[0]{\operatorname{ord}}	
\newcommand{\GL}[0]{\operatorname{GL}}
\newcommand{\supp}[0]{\operatorname{supp}}	
\newcommand{\id}[0]{\operatorname{id}}		
\newcommand{\Sp}[0]{\operatorname{Sp}}		
\newcommand{\eins}[0]{\mathbf{1}}			
\newcommand{\diag}[0]{\operatorname{diag}}
\newcommand{\ind}[0]{\operatorname{ind}}
\newcommand{\auf}[1]{\quad\stackrel{#1}{\longrightarrow}\quad}
\newcommand{\hull}[0]{\operatorname{hull}}
\newcommand{\prim}[0]{\operatorname{Prim}}
\newcommand{\ad}[0]{\operatorname{Ad}}
\newcommand{\quot}[0]{\operatorname{Quot}}
\newcommand{\ext}[0]{\operatorname{Ext}}
\newcommand{\ev}[0]{\operatorname{ev}}
\newcommand{\fin}[0]{{\subset\!\!\!\subset}}
\newcommand{\diam}[0]{\operatorname{diam}}
\newcommand{\Hom}[0]{\operatorname{Hom}}
\newcommand{\Aut}[0]{\operatorname{Aut}}
\newcommand{\del}[0]{\partial}
\newcommand{\dimeins}[0]{\dim^{\!+1}}
\newcommand{\dimnuc}[0]{\dim_{\mathrm{nuc}}}
\newcommand{\dimnuceins}[0]{\dimnuc^{\!+1}}
\newcommand{\dr}[0]{\operatorname{dr}}
\newcommand{\dimrok}[0]{\dim_{\mathrm{Rok}}}
\newcommand{\dimrokeins}[0]{\dimrok^{\!+1}}
\newcommand{\dreins}[0]{\dr^{\!+1}}
\newcommand*\onto{\ensuremath{\joinrel\relbar\joinrel\twoheadrightarrow}} 
\newcommand*\into{\ensuremath{\lhook\joinrel\relbar\joinrel\rightarrow}}  
\newcommand{\im}[0]{\operatorname{im}}
\newcommand{\dst}[0]{\displaystyle}
\newcommand{\cstar}[0]{$\mathrm{C}^*$}
\newcommand{\ann}[0]{\operatorname{Ann}}
\newcommand{\dist}[0]{\operatorname{dist}}
\newcommand{\idlat}[0]{\operatorname{IdLat}}
\newcommand{\Cu}[0]{\operatorname{Cu}}
\newcommand{\Ost}[0]{\CO_\infty^{\mathrm{st}}}
\newcommand{\linhull}{\operatorname{span}}

\newtheorem{satz}{Satz}[section]		
\newtheorem{cor}[satz]{Corollary}
\newtheorem{lemma}[satz]{Lemma}
\newtheorem{prop}[satz]{Proposition}
\newtheorem{theorem}[satz]{Theorem}
\newtheorem*{theoremoz}{Theorem}

\theoremstyle{definition}
\newtheorem{defi}[satz]{Definition}
\newtheorem*{defioz}{Definition}
\newtheorem{defprop}[satz]{Definition \& Proposition}
\newtheorem{nota}[satz]{Notation}
\newtheorem*{notaoz}{Notation}
\newtheorem{rem}[satz]{Remark}
\newtheorem*{remoz}{Remark}
\newtheorem{example}[satz]{Example}
\newtheorem{defnot}[satz]{Definition \& Notation}
\newtheorem{question}[satz]{Question}
\newtheorem*{questionoz}{Question}
\newtheorem{construction}[satz]{Construction}
\newtheorem{conjecture}[satz]{Conjecture}


\section{Introduction}

The Thompson groups $F$, $T$ and $V$ (see \cite{Hig}, \cite{CFP}) are among the most mysterious and most intensly studied discrete groups. 
We want to exploit the natural representation of these groups inside the unitary group of the Cuntz algebra $\OO_2$ (see \cite{B}, \cite{N}) 
and initiate a line of investigations aimed at better understanding of their internal symmetries provided by endomorphisms and automorphisms. 
It should be noted that this relation between the Thompson groups and the Cuntz algebras has been exploited recently by Uffe Haagerup and his 
collaborators in their work on amenability and other analytic properties of the Thompson groups, see \cite{HHRS} and \cite{HO}. 

The central question we ask in this paper is the following.
\begin{questionoz}
Which unital $*$-endomorphisms of $\OO_2$ preserve the Thompson groups globally?
\end{questionoz}
Recall from 
\cite{Cun2} that every such endomorphism of $\OO_2$ is of the form $\lambda_u$ for some unitary $u\in\U(\OO_2)$. Our main result, Theorem 
\ref{mainresult}, says that $\lambda_u(V)\subseteq V$ if and only if $u\in V$. Under the weaker assumptions that $\lambda_u(F)\subseteq V$ or 
$\lambda_u(T)\subseteq V$, we are not able to conclude that $u\in V$ without additional conditions, explained in Propositions \ref{fincl} and 
\ref{prop:T -> V}. However, as shown in Theorem \ref{mainauto}, if $\lambda_u(F)\subseteq V$ and $\lambda_u$ is an \emph{automorphism} of $\OO_2$ 
then the unitary $u$ must belong to group $V$. 

Note also that it is quite possible that endomorphism (or automorphism) $\lambda_u$ of $\OO_2$ globally preserves the Thompson group 
$F$, while the unitary $u$ does not belong to $F$ --- the flip-flop automorphism of $\OO_2$ is one such example. The non-trivial combinatorial question 
of determining those unitaries $u\in V$ for which  $\lambda_u(F)\subseteq F$ is taken up in \cite{BRS}. 

In the course of these investigations, we have discovered a useful technical condition on endomorphisms of $\OO_n$, which we call {\em modest 
scaling} (Definition \ref{modestscaling}). This condition is automatically satisfied by all automorphisms of $\OO_n$ as well as by those unital
endomorphisms preserving the core UHF-subalgebra $\F_n$ of $\OO_n$ (Proposition \ref{prop:ex mod scal endos}). 
Modest scaling is a non-commutative analogue of the topological property of a continuous surjection of a compact space that the preimage of every point 
has empty interior (Remark \ref{mstop}).


\section{Preliminaries}

\subsection{The Thompson groups}

The Thompson group $F$ is the group of order preserving piecewise linear homeomorphisms of the closed 
interval $[0,1]$ onto itself that are differentiable except finitely many dyadic rationals and such that all slopes 
are integer powers of $2$. This group has a presentation 
$$
F=\langle x_0,x_1,\ldots,x_n,\ldots \;|\; x_j x_i=x_i x_{j+1}, \forall i<j\rangle. 
$$
Remarkably, it admits a finite presentation as well
$$
F=\langle A,B \;|\; [AB^{-1},A^{-1}BA]=1, \; [AB^{-1}, A^{-2}BA^2]=1\rangle.
$$
These two are connected by setting $x_0=A$ and $x_n=A^{-(n-1)}BA^{n-1}$ 
for $n\geq 1$. 

We will also consider the Thompson groups $T$ and $V$. The latter consists of possibly discontinuous  bijections 
of the interval $[0,1]$ that are piecewise linear with slopes integer powers of $2$ and with finitely points 
of discontinuity and non-differentiability that are all diadic rationals. We have $F\subseteq T\subseteq V$. 

For a good introduction to the Thompson groups we refer the reader to \cite{CFP} and \cite{BB} and the references therein. 


\subsection{The Cuntz algebra $\OO_n$} \label{prelim:Cuntz alg}

If $n$ is an integer greater than 1, then the Cuntz algebra $\OO_n$ is a $C^*$-algebra generated by $n$ isometries $S_1, \ldots, S_n$ satisfying
$\sum_{i=1}^n S_i S_i^* = 1$. It is simple, purely infinite, so that its isomorphism type does not depend on the choice of isometries, \cite{Cun}.
We denote by $W_n^k$ the set of $k$-tuples $\mu = \mu_1\ldots\mu_k$
with $\mu_m \in \{1,\ldots,n\}$, and by $W_n$ the union $\cup_{k=0}^\infty W_n^k$,
where $W_n^0 = \{\emptyset\}$. We call elements of $W_n$ multi-indices.
If $\mu \in W_n^k$ then $|\mu| = k$ is the length of $\mu$. 
If $\mu = \mu_1\ldots\mu_k \in W_n$ then $S_\mu = S_{\mu_1} \ldots S_{\mu_k}$
($S_\emptyset = 1$ by convention) is an isometry with range projection $P_\mu=S_\mu S_\mu^*$. 
We say that $\mu,\nu\in W_n$ are orthogonal if $P_\mu P_\nu=0$. 
Every word in $\{S_i, S_i^* \ | \ i = 1,\ldots,n\}$ can be uniquely expressed as
$S_\mu S_\nu^*$, for $\mu, \nu \in W_n$ \cite[Lemma 1.3]{Cun}.

We denote by $\F_n^k$ the $C^*$-subalgebra of $\OO_n$ spanned by all words of the form
$S_\mu S_\nu^*$, $\mu, \nu \in W_n^k$, which is isomorphic to the
matrix algebra $M_{n^k}({\mathbb C})$. The norm closure $\F_n$ of
$\cup_{k=0}^\infty \F_n^k$ is isomorphic to the UHF-algebra of type $n^\infty$,
called the {\em core UHF-subalgebra} of $\OO_n$, \cite{Cun}. We denote by $\tau$ the 
unique normalized trace on $\F_n$. The core UHF-subalgebra $\F_n$ is the fixed-point 
algebra for the gauge action $\gamma:U(1)\to\aut(\OO_n)$, such that $\gamma_z(S_j)=zS_j$ for 
$z\in U(1)$ and $j=1,\ldots,n$. We denote by $\Phi_\F$ the faithful conditional expectation from 
$\OO_n$ onto $\F_n$ given by averaging with respect to the normalized Haar measure: 
$$ \Phi_\F(x) = \int_{z\in U(1)}\gamma_z(x)dz. $$

The $C^*$-subalgebra of $\OO_n$ generated by projections $P_\mu$, $\mu\in W_n$, is a 
MASA (maximal abelian $*$-subalgebra) in $\OO_n$. We call it the {\em diagonal} and denote $\D_n$. 
The spectrum of $\D_n$ is naturally identified with $X_n$ --- the full one-sided $n$-shift space. Furthermore, 
there exists a unique  faithful conditional expectation $\Phi_\D$ from $\OO_n$ onto $\D_n$ such that 
$$ \Phi_\D=\Phi_\D\circ\Phi_F $$ 
and $\Phi_\D(S_\alpha S_\beta^*) = 0$ for all $\alpha, \beta \in W_n$ such that $\alpha \neq \beta$.

We need to introduce the following notation, for use later in the paper. 
Let $q \in \CD_n$ be a non-zero projection. Then $q$ admits a unique representation $q = \sum_{j=1}^r P_{\mu_j}$ with minimal $r \geq 1$ among all representations for $q$ as a finite sum of projections of the form $P_\mu$, $\mu \in W_n$. In the following, we call this representation  the \emph{standard form} for $q$. We set
\[
\min(q) := \min\set{|\mu_j| \mid q = \sum\limits_{j=1}^r P_{\mu_j} \text{ standard form}}.
\]

In what follows, we will consider elements of $\OO_n$ of the form $w=\sum_{(\alpha,\beta)\in\J}
S_\alpha S_\beta^*$, where $\J$ is a finite collection  of pairs $(\alpha,\beta)$  in $W_n\times W_n$. The collection of all 
such elements will be denoted $\CV_n$, that is 
\[
\CV_n = \set{w \in \CO_n \mid w = \sum_{(\alpha,\beta)\in\J}S_\alpha S_\beta^*}.
\]
Clearly, $\CV_n$ is a $*$-subring of $\CO_2$.
We put $\J_1=\{\alpha \mid \exists (\alpha,\beta)\in\J\}$ and $\J_2=\{\beta \mid \exists (\alpha,\beta)\in\J\}$. 
Of course, such a presentation of an element of $\CV_n$ is not unique. 

We denote by $\SS_n$ the group of unitaries in $\CV_n$, that is those unitaries in $\OO_n$ of the form $\sum_{(\alpha,\beta)\in\J}S_\alpha S_\beta^*$. Such a sum is a unitary if and only if $\sum_{\alpha\in\J_1} P_\alpha = 1 = \sum_{\beta\in\J_2} P_\beta$. It is easy to see that $\SS_n$ is contained in the normalizer of $\D_n$ in $\OO_n$,
$$
{\mathcal N}_{\OO_n}(\D_n)=\{u\in\U(\OO_n) \mid u\D_n u^*=\D_n\}.
$$ 

For a unital $*$-subalgebra $A$ of $\OO_n$, we denote by $\U(A)$ the group of unitary elements of $A$ and by $\CP(A)$ the set of projections in $A$.


\subsection{Endomorphisms of  $\OO_n$} 

As it is shown by Cuntz in \cite{Cun2}, there exists a bijective correspondence
between unitaries in $\OO_n$ and unital $*$-endomorphisms 
of $\OO_n$, determined by
$$ \lambda_u(S_i) = u S_i, \;\;\; i=1,\ldots, n. $$ 
Such maps $\lambda_u$ will be called endomorphisms for short, and the collection of all of them will be denoted $\en(\OO_n)$. Note that 
composition of endomorphisms corresponds to the `convolution' multiplication of unitaries:
$\lambda_u \circ \lambda_w = \lambda_{\lambda_u(w)u}$.
In the case $u,w\in\U(\F_n^1)$ or $u,w\in\U(\D_n)$,  this formula simplifies to $\lambda_u\circ\lambda_w=\lambda_{uw}$. 
For all $u\in\U(\OO_n)$ we have $\Ad(u)=\lambda_{u\varphi(u^*)}$. Here $\varphi$ denotes  the canonical shift on the Cuntz algebra:
$$ \varphi(x)=\sum_{i=1}^n S_ixS_i^*, \;\;\; x\in\OO_n. $$


\subsection{Representations of the Thompson groups in $\U(\OO_2)$}

As shown in \cite{B} and \cite{N}, the Thompson group $F$ has a natural faithful representation in the unitary 
group of $\OO_2$ by those  unitaries $u=\sum_{(\alpha,\beta)\in\J}S_\alpha S_\beta^*$ 
in $\SS_2$ that the association $\J_1\ni\alpha\mapsto\beta\in\J_2$ (with $(\alpha,\beta)\in\J)$ 
respects the lexicographic order on $W_2$.  We have 
\[ \begin{aligned}
x_0 & = S_1S_1S_1^* + S_1S_2S_1^*S_2^* + S_2S_2^*S_2^*, \\
x_k & = 1-S_2^kS_2^{*k} + S_2^kx_0S_2^{*k}, \;\;\; \text{for } k\geq 1. 
\end{aligned} \]

The subgroup of $\SS_2$ generated by $F$ and $S_2S_2S_1^* + S_1S_1^*S_2^* + S_2S_1S_2^*S_2^*$ 
is isomorphic to the Thompson group $T$, and consists of those unitaries 
$u=\sum_{(\alpha,\beta)\in\J}S_\alpha S_\beta^*$ 
in $\SS_2$ that the association $\J_1\ni\alpha\mapsto\beta\in\J_2$ (with $(\alpha,\beta)\in\J)$ 
respects the lexicographic order on $W_2$ up to a cyclic permutation. 

Finally, group $\SS_2$ itself is isomorphic to the Thompson group $V$ and it will be denoted in this way throughout the remainder of this paper. 
We have $V = \CV_2 \cap \CU(\CO_2)$. 

We note that the Thompson group $F$ is invariant under the canonical shift $\varphi$ on $\OO_2$. Furthermore, it is quite possible that 
$\lambda_u(F)=F$ for an automorphism $\lambda_u$ of $\OO_2$, even though the unitary $u$ may not belong to group $F$. The simplest example 
is the flip-flop automorphism $\lambda_u(S_1)=S_2$, $\lambda_u(S_2)=S_1$, where the corresponding unitary $u=S_1S_2^* +S_2S_1^*$ is inside  
group $T$ but not in $F$. 


\section{The main results}

\subsection{Modestly scaling endomorphisms}

In this subsection, we introduce a certain class of endomorphisms of $\OO_n$, called modestly scaling, that will also play a role in our considerations 
of endomorphisms preserving the Thompson groups. 

\begin{defi}\label{modestscaling}
An endomorphism $\alpha \in \en(\CO_n)$ is called \emph{modestly scaling} if the following property is satisfied. For every sequence $(\nu_k)$  
of non-empty multi-indices in $W_n$, if $p \in {\mathcal P}(\CO_n)$ such that $p \leq \alpha(P_{\nu_1\ldots\nu_k})$ for all $k \in \IN$ then $p=0$.
\end{defi}

\begin{rem}\label{mstop}\rm
Suppose that $\alpha\in\en(\OO_n)$ is such that $\alpha(\D_n)\subseteq\D_n$. Let $\alpha_*:X_n\to X_n$ be the corresponding 
continuous surjection of the spectrum of $\D_n$. If $\alpha$ is modestly scaling then for every point $x\in X_n$ the inverse image 
$\alpha_*^{-1}(x)$ has an empty interior or,  equivalently, $\alpha_*$ is not constant on any open subset of $X_n$. 
\end{rem}

\begin{prop} \label{prop:ex mod scal endos}
Let $\alpha\in\en(\OO_n)$. Then $\alpha$ is modestly scaling if one of the following conditions holds:
\begin{itemize}
 \item[(i)] $\alpha$ is an automorphism of $\OO_n$;
 \item[(ii)] $\alpha(\CF_n)\subseteq\CF_n$.
\end{itemize}
\end{prop}
\begin{proof}
Let $(\nu_k)$ be a sequence of non-empty multi-indices in $W_n$ and set $\mu_k = \nu_1\ldots\nu_k$. 
Let $p \in \CO_n$ be a projection such that $p \leq \alpha(P_{\mu_k})$ for all $k \in \IN$.

Ad (i). Let $\alpha\in\aut(\OO_n)$. Since $\alpha^{-1}(p) \leq P_{\mu_k}$ for all $k \in \IN$, we have 
$\Phi_\D(\alpha^{-1}(p)) \leq P_{\mu_k}$ for all $k \in \IN$. Thus $\Phi_\D(\alpha^{-1}(p)) = 0$ and therefore also $p = 0$, as $\Phi_\D$ is faithful. 

Ad (ii). Due to the uniqueness of trace on the UHF-algebra $\F_n$, we have 
\[
\tau(\Phi_\F(p)) \leq \tau(\alpha(P_{\mu_k})) = \tau(P_{\mu_k})
\]
for all $k\in\IN$. 
As $\tau(P_{\mu_k}) \stackrel{k \to \infty}{\longrightarrow} 0$ and $\tau$ is faithful, we get that $\Phi_\F(p) = 0$. Hence $p=0$, as $\Phi_\F$ 
is faithful. 
\end{proof}
\begin{rem}\rm
As shown already by Cuntz in \cite{Cun2}, if $u$ is a unitary inside the core UHF-subalgebra $\F_n$ then automatically the corresponding 
endomorphism $\lambda_u$ globally preserves $\F_n$. However, it should be noted that there exist unitaries $u\in\U(\OO_n)$ which do not belong 
to $\F_n$  for which nevertheless we have $\lambda_u(\F_n)\subseteq\F_n$. Such exotic endomorphisms of $\OO_n$ have been thoroughly 
investigated in \cite{CRS}, \cite{Hay} and \cite{HHSz}. 
\end{rem}

For those endomorphisms which globally preserve the diagonal MASA $\D_n$, the following proposition gives a useful criterion of modest scaling 
for an endomorphism.  Recall for this the definition of $\min(q)$ from Subsection~\ref{prelim:Cuntz alg} for a non-trivial projection $q \in \D_n$.

\begin{prop} \label{lem: non-examples mod scal}
Let $\alpha\in\en(\CO_n)$ be such that $\alpha(\CD_n) \subseteq \CD_n$. Then the following two conditions are equivalent:
\begin{itemize}
\item[(i)] endomorphism $\alpha$ is modestly scaling;
\item[(ii)] for every sequence of non-empty multi-indices $(\nu_k)$ in $W_n$ we have 
$$ \min(\alpha(P_{\nu_1\ldots\nu_k})) \stackrel{k \to \infty}{\longrightarrow} \infty. $$
\end{itemize}
\end{prop}
\begin{proof}
Let $(\nu_k)$ be a sequence of non-empty multi-indices in $W_n$. Set $\mu_k = \nu_1\ldots\nu_k$. 

(i)$\Rightarrow$(ii). Assume  that $\alpha$ is modestly scaling. Suppose, by way of contradiction, that sequence $\min(\alpha(P_{\mu_k}))$, $k \in \Nb$ is bounded. Observe that this sequence is monotonely increasing, as $\alpha(P_{\mu_{k+1}}) \leq \alpha(P_{\mu_k})$ for all $k$, and therefore eventually stabilizes. It is not difficult to see that there exists a non-empty multi-index $\kappa \in W_n$ such that $|\kappa|=\sup_k\{\min(\alpha(P_{\mu_k}))\}$ 
and with the property that $P_\kappa \leq \alpha(P_{\mu_k})$ for all $k$. This contradicts the fact that $\alpha$ is modestly scaling.

(ii)$\Rightarrow$(i). Let $p \in \CO_n$ be a projection such that $p \leq \alpha(P_{\mu_k})$ for all $k$. Hence  
$0 \leq \Phi_\D(p) \leq \alpha(P_{\mu_k})$. Assume that $\Phi_\D(p) \neq 0$ and find some $\sigma \in W_n$ and $t > 0$ such that 
$t P_\sigma \leq \Phi_\D(p)$. By assumption, $\min(\alpha(P_{\mu_k})) > |\sigma|$ for sufficiently large $k$, so that $t P_\sigma \leq \alpha(P_{\mu_k})$ is not possible for such $k$. This is a contradiction. Hence $\Phi_\D(p) = 0$ and thus $p = 0$ by faithfulness of $\Phi_\D$. 
\end{proof}

There exist endomorphisms of $\CO_n$ which are not modestly scaling, and the following proposition provides one way for constructing 
such examples. Here for $k \geq 1$ and $i = 1,\ldots, n$, we denote by $i(k) \in W_n^k$ the multi-index of length $k$ consisting only of $i$'s.

\begin{prop}\label{notmodsc}
Let $k \geq 1$ and $i \in \{1,\ldots,n\}$. Let $v \in \CO_n$ be a partial isometry with $v^*v = 1- P_{i(k+1)}$ and $vv^* = 1- P_{i(k)}$. We put  
$w = v + S_{i(k)}S_{i(k+1)}^*$, a unitary in $\CO_n$. Then $P_{i(k)} \leq \lambda_w(P_{i(r)})$ for all $r \geq 1$. In particular, $\lambda_w$ is not modestly scaling.
\end{prop}
\begin{proof}
The proof is by induction on $r \geq 1$. Using $vP_{i(k+1)} = 0$, one computes for $r = 1$:
\[
\begin{aligned}
\lambda_w(P_i) & =  wP_iw^* \\
 & =  vP_iv^* + vP_i S_{i(k+1)}S_{i(k)}^* + S_{i(k)}S_{i(k+1)}^*P_iv^* +  S_{i(k)}S_{i(k+1)}^*P_i S_{i(k+1)}S_{i(k)}^* \\
 & =  vP_iv^* + P_{i(k)}.
\end{aligned}
\]
Thus, $P_{i(k)} \leq \lambda_w(P_i)$.
Assume now that $P_{i(k)} \leq \lambda_w(P_{i(r)})$. Then there exists some projection $p \in \CO_n$ such that $\lambda_w(P_{i(r)}) = p + P_{i(k)}$. Using again that $vP_{i(k+1)} = 0$, we compute
\[
\begin{array}{lcl}
\lambda_w(P_{i(r+1)}) & = & wS_i\lambda_w(P_{i(r)})S_i^*w^*  \\
& = & wS_ipS_i^*w^* + wS_iP_{i(k)}S_i^*w^* \\
& = & wS_ipS_i^*w^* + S_{i(k)}S_{i(k+1)}^*P_{i(k+1)}S_{i(k+1)}S_{i(k)}^* \\
& = & wS_ipS_i^*w^* + P_{i(k)}. 
\end{array}
\]
Thus $P_{i(k)} \leq \lambda_w(P_{i(r)})$ for all $r\geq 1$, by the principle of mathematical induction. Consequently, $\lambda_w$ 
does not satisfy Definition \ref{modestscaling} and hence it is not modestly scaling. 
\end{proof}

Proposition \ref{notmodsc} shows that many unitaries in $\CS_n$ yield endomorphisms that are not modestly scaling. Particular instances 
are the generators $x_k$, $k \geq 0$, of the Thompson group $F$. It is conceivable that many unitaries in $F$ lead to endomorphisms which are not modestly scaling. However, one should notice that for any $u \in F$, the element $u\phi(u)^* \in F$ corresponds to an inner automorphism $\lambda_{u\phi(u)^*}=\Ad(u)$ of $\CO_2$, which is modestly scaling.

\begin{lemma} \label{lem:equal or orth}
Let $\alpha\in\en(\CO_n)$, $x \in \CO_n$, and $\mu, \nu \in W_n$ be such that $x\alpha(S_\mu) = x\alpha(S_\nu) \neq 0$. Then $P_\mu P_\nu \neq 0$. Moreover, if $\alpha$ is modestly scaling then $\mu = \nu$.
\end{lemma}
\begin{proof}
Let $x\alpha(S_\mu) = x\alpha(S_\nu) \neq 0$. Then 
\[
0 \neq x\alpha(S_\mu S_\mu^*)x^* =x\alpha(S_\nu S_\nu^*)x^* \geq 0,
\]
and hence, 
\[
0 \neq x\alpha(P_\mu)x^*x\alpha(P_\nu)x^* \leq \Vert x^*x \Vert x\alpha(P_\mu P_\nu)x^*.
\]
This shows that $P_\mu P_\nu \neq 0$. 

Assume now that $\alpha$ is modestly scaling. Since $P_\mu P_\nu \neq 0$, we may assume without loss of generality that there exists a $\kappa \in W_n$ such that $\nu = \mu \kappa$. Also, suppose by way of contradiction, that $\mu \neq \nu$, that is, $\kappa \neq \emptyset$. Set $y = \alpha(S_\mu)^* x^*x\alpha(S_\mu)$ and observe that $y = \alpha(P_{\mu\kappa^\ell})y\alpha(P_{\mu\kappa^\ell})$ for all $\ell \in \IN$. Here $\kappa^\ell \in W_n$ denotes the concatenation of $\ell$ copies of $\kappa$. As the Cuntz algebra $\CO_n$ is purely infinite and simple, we find a non-zero projection $p \in \CO_n$ with $||y||p \leq y$. Then $p \leq \alpha(P_{\mu\kappa^\ell})$ for all $\ell \in \IN$. As $\alpha$ is modestly scaling, we conclude that $p = 0$, which is a contradiction. This yields $\mu = \nu$, and the proof is complete.
\end{proof}

\begin{defi}
For $w \in \CS_n$ we denote $\eins_w = \Phi_\D(w)$. 
It is easy to see that $\eins_w$ is the maximal projection in $\CD_n$ such that $w = \eins_w + (1 - \eins_w)w$. 
\end{defi}

\begin{prop} \label{prop:alpha(eins_w)}
Let $\alpha\in\en(\CO_n)$ and $w \in \CS_n$ be such that $\alpha(w) \in \CS_n$. Then $\alpha(\eins_w) \leq \eins_{\alpha(w)}$. If $\alpha$ is modestly scaling then $\alpha(\eins_w) = \eins_{\alpha(w)}$.
\end{prop}
\begin{proof}
As $\alpha(w) \in \CS_n$, there exists a finite set $\tilde{\J}\subseteq W_n\times W_n$ such that 
$\alpha(w) = \sum_{(\mu,\nu)\in\tilde{\J}} S_{\mu}S_{\nu}^*$. We have that 
\[
\alpha(\eins_w) = \alpha(\eins_w)\alpha(w) = \alpha(\eins_w)\sum_{(\mu,\nu)\in\tilde{\J}} S_{\mu}S_{\nu}^*,
\]
and thus, $\alpha(\eins_w)S_{\mu} = \alpha(\eins_w)S_{\nu}$ for $(\mu,\nu)\in\tilde{\J}$. By Lemma~\ref{lem:equal or orth}, 
we therefore get that if $\alpha(\eins_w)S_{\mu} \neq 0$, then $\mu = \nu$.  Thus, $\alpha(\eins_w) \leq \eins_{\alpha(w)}$.

Now assume that $\alpha$ is modestly scaling. Write $w = \sum_{(\kappa,\sigma)\in\J}S_{\kappa}S_{\sigma}^*$ for some finite set
$\J\subseteq W_n\times W_n$. Using that $\alpha(w) \in \CS_n$, we can argue as before to deduce $\eins_{\alpha(w)}\alpha(S_{\kappa}) = \eins_{\alpha(w)}\alpha(S_{\sigma})$ for all $(\kappa,\sigma)\in\J$. It follows from Lemma~\ref{lem:equal or orth} that $\eins_{\alpha(w)}\alpha(S_{\kappa})\neq 0$ implies $\kappa= \sigma$. This shows that $\eins_{\alpha(w)} \leq \alpha(\eins_w)$.
\end{proof}


\subsection{Endomorphisms globally preserving the Thompson groups}

In this subsection, we investigate which endomorphisms of $\OO_2$ globally preserve the Thompson groups, in terms of the corresponding 
unitaries of $\OO_2$. 

\begin{rem} \label{rem:eins_w}
We notice that 
for any given projection $p \in \CD_2$ there exists a unitary $w \in F$ such that $p = \eins_w$. Indeed, if $1 - p = \sum_{j=1}^m P_{\mu_j}$, then $w = p + \sum_{j=1}^m S_{\mu_j}x_0S_{\mu_j}^* \in F$ satisfies
\[
\eins_w = \Phi_\D(w) = p + \sum_{j=1}^m \Phi_\D(S_{\mu_j}x_0S_{\mu_j}^*) = p.
\]
\end{rem}

Combining Proposition~\ref{prop:alpha(eins_w)} and Remark~\ref{rem:eins_w}, we immediately obtain the following.

\begin{cor} \label{cor:char alpha(D_2) in D_2}
Let $\alpha\in\en(\CO_2)$ be modestly scaling and such that $\alpha(F) \subseteq V$. Then  $\alpha(\eins_w) = 
\eins_{\alpha(w)}$ for all $w \in F$, and hence $\alpha(\CD_2) \subseteq \CD_2$. 
\end{cor}

Now, we are ready to prove the first main result of this paper. 

\begin{theorem}\label{mainauto}
Let $u\in\CU(\CO_2)$ be such that $\lambda_u\in\Aut(\CO_2)$ and $\lambda_u(F)\subseteq V$.  
Then $u\in V$. 
\end{theorem}
\begin{proof}
Let $u\in\CU(\CO_2)$ be such that $\lambda_u\in\Aut(\CO_2)$ and $\lambda_u(F)\subseteq V$. Then $\lambda_u(\CD_2)=\CD_2$ by 
Proposition \ref{prop:ex mod scal endos} and Corollary~\ref{cor:char alpha(D_2) in D_2}. 
Thus $u\in\CN_{\CO_2}(\CD_2)$ by \cite{Cun2}. As shown in \cite{P} and \cite{CS}, this implies that $\lambda_u=\lambda_v \circ \lambda_d$ for some $v\in V$ and $d\in\CU(\CD_2)$ with both $\lambda_v$ and $\lambda_d$ automorphisms of $\CO_2$.

Suppose, by way of contradiction, that $d\neq1$. Let $t \neq 1$ be a scalar of modulus one in the spectrum of $d$. Let $\eps > 0$ be such that $|n - t| \geq \eps$ for all non-negative integers $n$. Find a non-empty multi-index $\beta \in W_2$ and a partial unitary $x \in \CD_2$ with support and range projection $1 - P_\beta$ such that $\| tP_\beta +x - d \| < \eps$.

Denote by $\tilde{\beta} \in W_2$ the unique multi-index such that $\beta = \beta_1 \tilde{\beta}$ with $|\beta_1|=1$. It is not difficult to show that there exists some partial isometry $y \in \CV_2$ such that $w = S_{\beta 12}S_{\tilde{\beta}12}^* + y \in F$. One computes
\[
\lambda_d(w P_{\tilde{\beta}12}) = \lambda_d(S_{\beta 12}S_{\tilde{\beta}12}^*) =  \lambda_d(S_{\beta_1}P_{\tilde{\beta}12}) = dS_{\beta_1}P_{\tilde{\beta}12}
= d S_{\beta 12}S_{\tilde{\beta}12}^*,
\]
from which it follows that $\| \lambda_d(w P_{\tilde{\beta}12}) - tS_{\beta 12}S_{\tilde{\beta}12}^*\| < \eps$. Hence,
\[
 \| \lambda_u(w P_{\tilde{\beta}12}) -  t v S_{\beta_1} \lambda_v(P_{\tilde{\beta}12}) \| = \| \lambda_v(\lambda_d(w P_{\tilde{\beta}12}) - tS_{\beta 12}S_{\tilde{\beta}12}^*) \| < \eps.
\]
As $v \in V$, it holds that $y := v S_{\beta_1} \lambda_v(P_{\tilde{\beta}12}) \in \CV_2$. Find some finite set $\J \subseteq W_2 \times W_2$ such that $y = \sum_{(\alpha,\beta) \in \J} S_\alpha S_\beta^*$. For $(\alpha,\beta) \in \J$, we have
\[
\begin{array}{lcl}
\|\Phi_{\CD}(S_\beta S_\alpha^*\lambda_u(wP_{\tilde{\beta}12})) - tP_\beta \| & = & \| \Phi_{\CD}(S_\beta S_\alpha^*\lambda_u(wP_{\tilde{\beta}12})) - \Phi_{\CD}(tS_\beta S_\alpha^*y) \| \\
 & \leq & \| S_\beta S_\alpha^*(\lambda_u(wP_{\tilde{\beta}12})  - ty) \| \\
 & < & \eps.
\end{array}
\]

Let $\chi:\CD_2 \to \IC$ be any character satisfying $\chi(P_\beta) = 1$. Then
$$
|\chi(\Phi_\D(S_\beta S_\alpha^*\lambda_u(wP_{\tilde{\beta}12}))) - t | < \eps.
$$
By the choice of $\eps > 0$, we get that $\chi(\Phi_\D(S_\beta S_\alpha^*\lambda_u(wP_{\tilde{\beta}12})))$ cannot be a non-negative integer. On the other hand, $S_\beta S_\alpha^*\lambda_u(wP_{\tilde{\beta}12}) \in \CV_2$ by assumption. Hence, $\Phi_{\CD}(S_\beta S_\alpha^*\lambda_u(wP_{\tilde{\beta}12})) \in \CV_2 \cap \CD_2$ is a finite sum of projections in $\CD_2$. This implies that $\chi(\Phi_{\CD}(S_\beta S_\alpha^*\lambda_u(wP_{\tilde{\beta}12})))$ is a non-negative integer, which is a contradiction. Hence, $d = 1$ and thus $u = v \in V$, as required.
\end{proof}

By Theorem \ref{mainauto} above, if $\alpha\in\aut(\OO_2)$ restricts to an automorphism of one of the Thompson 
groups $F$, $T$, or $V$, then $\alpha=\lambda_u$ for some $u\in V$.  

\begin{lemma} \label{lem:orth proj alpha}
Let $\alpha\in\en(\CO_2)$ and $\mu,\nu \in W_2$ be non-empty, orthogonal multi-indices. Assume there exists some $w \in V$ such that $P_\mu \leq \eins_w$, $P_\nu wP_\nu = 0$, and $\alpha(w) \in V$. Then $\Phi_\D(\alpha(P_\mu))\alpha(P_\nu) = 0$.
\end{lemma}
\begin{proof}
Write $w = \sum_{(\kappa,\sigma) \in \J} S_\kappa S_\sigma^*$ for some finite set $\J \subseteq W_2 \times W_2$. Without loss of generality we may assume that there exists $\CG \subseteq \J_2$ such that $P_\nu = \sum_{\sigma \in \CG} P_\sigma$. Let $\sigma \in \CG$ and $\kappa \in \J_1$ be the unique multi-index such that $(\kappa,\sigma) \in \J$. Then $\eins_{\alpha(w)}\alpha(S_\sigma) = \eins_{\alpha(w)}\alpha(S_\kappa)$ and therefore Lemma~\ref{lem:equal or orth} yields that $\eins_{\alpha(w)}\alpha(P_\sigma) = 0$. Here we use that $P_\kappa P_\sigma = 0$ by assumption. Thus,
$$
\eins_{\alpha(w)}\alpha(P_\nu) =  \sum_{\sigma \in \CG} \eins_{\alpha(w)}\alpha(P_\sigma) = 0.
$$
On the other hand, $\alpha(P_\mu) \leq \alpha(\eins_w) \leq \eins_{\alpha(w)}$ by Proposition~\ref{prop:alpha(eins_w)}. Therefore $\Phi_\D(\alpha(P_\mu)) \leq \eins_{\alpha(w)}$ 
and consequently $\Phi_\D(\alpha(P_\mu))\alpha(P_\nu) = 0$.
\end{proof}

\begin{rem} \label{rem:w in S_2}
Let $\mu,\nu \in W_2$ be non-empty, orthogonal multi-indices such that $\min(|\mu|,|\nu|) \geq 2$. It is easy to see that there exists $w \in V$ with the property that $P_\mu \leq \eins_w$ and $P_\nu w P_\nu = 0$.
\end{rem}

In general, a unitary $w \in V$ as in Remark~\ref{rem:w in S_2} cannot be chosen inside the Thompson group $F$ or $T$. However, the following observation shows that this is still possible for many choices of $(\mu, \nu) \in W_2 \times W_2$. In the following Lemma \ref{lem:special element F}, 
we write $\mu\prec\nu$ to indicate that $\mu$ precedes $\nu$ in the lexicographic order. Recall that for $k \geq 1$ and $i = 1,2$, we denote by $i(k) \in W_2^k$ the multi-index of length $k$ consisting only of $i$'s.

\begin{lemma}  \label{lem:special element F}
Let $\mu, \nu \in W_2$ be non-empty, orthogonal multi-indices with $\mu \prec \nu$. Assume that there exists $\kappa \in W_2$ with $\mu \prec \kappa \prec \nu$ such that $\kappa$ is orthogonal to both $\mu$ and $\nu$. If $\nu \neq 2(k)$ for any $k \geq 1$, then there exists some $w \in F$ such that $P_\mu \leq \eins_w$ and $P_\nu w P_\nu = 0$. Similarly, if $\mu \neq 1(k)$ for any $k \geq 1$, then there exists some $w \in F$ such that $P_\nu \leq \eins_w$ and $P_\mu w P_\mu = 0$.
\end{lemma}
\begin{proof}
Assume first that $\nu \neq 2(k)$ for any $k \geq 1$. We find projections $p_1,p_2,p_3,p_4 \in \CD_2$ such that
\begin{itemize}
\item[1)] $P_\mu$ + $P_\nu + \sum_{i=1}^4 p_i = 1$;
\item[2)] $p_1 = 0$ or $p = \sum_{j = 1}^{n_1} P_{\eta_j^{(1)}}$ for some multi-indices $\eta^{(1)}_1,\ldots,\eta^{(1)}_{n_1} \prec \mu$;
\item[3)] $0 \neq p_2 = \sum_{j = 1}^{n_2} P_{\eta_j^{(2)}}$ for some multi-indices $\mu \prec \eta^{(2)}_1,\ldots,\eta^{(2)}_{n_2} \prec \nu$;
\item[4)] $0 \neq p_3 = \sum_{j = 1}^{n_3} P_{\eta_j^{(3)}}$ for some multi-indices $\nu \prec \eta^{(3)}_1,\ldots,\eta_{n_3}^{(3)}$;
\item[5)] $0 \neq p_4 = \sum_{j = 1}^{n_4} P_{\eta_j^{(4)}}$ such that $\eta_i^{(3)} \prec \eta_j^{(4)}$ for all $1 \leq i \leq n_3$ and $1 \leq j \leq n_4$.
\end{itemize}
Then one checks that there is some $w \in F$ with  
\begin{itemize}
\item[i)] $p_1 + P_\mu \leq \eins_w$;
\item[ii)] $wp_2 = (p_2 + P_\nu)w$;
\item[iii)] $wP_\nu = p_3w$;
\item[iv)] $w(p_3 + p_4) = p_4w$.
\end{itemize}
In particular, it follows that $P_\mu \leq \eins_w$ and $P_\nu w P_\nu = P_\nu p_3 w= 0$.

If $\mu \neq 1(k)$ for any $k \geq 1$, then a similar proof shows that there exists some $w \in F$ such that $P_\nu \leq \eins_w$ and $P_\mu w P_\mu = 0$.
\end{proof}

\begin{lemma} \label{lem:mult dom}
Let $\alpha\in\en(\CO_2)$. Then $\alpha(\CD_2) \subseteq \CD_2$ if and only if for all $\mu,\nu \in W_2$ non-empty, orthogonal multi-indices it holds that
$\Phi_\D(\alpha(P_\mu))\Phi_\D(\alpha(P_\nu)) = 0$.
\end{lemma}
\begin{proof}
As the ``only if"-part is trivial, we only proof the ``if"-direction. Let $\mu \in W_2$ be a non-empty multi-index and find $\kappa_1,\ldots,\kappa_r \in W_2$ non-empty such that $P_\mu + \sum_{i=1}^r P_{\kappa_i} = 1$. Hence,
$$ \Phi_\D(\alpha(P_\mu)) + \sum_{i=1}^r \Phi_\D(\alpha(P_{\kappa_i})) = 1. $$ 
Multiplying this equation with $\Phi_\D(\alpha(P_\mu))$ and employing the assumption, we obtain that $\Phi_\D(\alpha(P_\mu)) = \Phi_\D(\alpha(P_\mu))^2$. This shows that $\alpha(P_\mu)$ belongs to the multiplicative domain of $\Phi_\D$. As $\Phi_\D$ is a faithful conditional expectation onto $\CD_2$, its multiplicative domain equals $\CD_2$. This concludes the proof.
\end{proof}

\begin{lemma} \label{prop:V->V implies D_2 -> D_2}
Let $\alpha\in\en(\OO_2)$. If $\alpha(V) \subseteq V$ then $\alpha(\CD_2) \subseteq \CD_2$.
\end{lemma}
\begin{proof}
Lemma~\ref{lem:orth proj alpha} combined with Remark~\ref{rem:w in S_2} shows that $\Phi_\D(\alpha(P_\mu))\alpha(P_\nu) = 0$ 
for all $\mu, \nu \in W_2$ non-empty, orthogonal multi-indices such that $\max(|\mu|,|\nu|) \geq 2$. However, this implies that 
$\Phi_\D(\alpha(P_\mu))\Phi_\D(\alpha(P_\nu)) = 0$ for all $\mu, \nu \in W_2$ non-empty and orthogonal. 
The conclusion now follows from Lemma~\ref{lem:mult dom}.
\end{proof}

Although, at this point, it is not clear whether the same conclusion holds if we only assume that $\alpha(F) \subseteq V$, we can at least say the following.

\begin{prop}
Let $\alpha\in\en(\CO_2)$ be  such that $\alpha(F) \subseteq V$. Then for any $\mu,\nu,\kappa \in W_2$ non-empty, mutually orthogonal multi-indices, 
we have 
\[
\Phi_\D(\alpha(P_\mu))\Phi_\D(\alpha(P_\nu))\Phi_\D(\alpha(P_\kappa)) = 0.
\]
\end{prop}
\begin{proof}
The claim follows directly from Lemma~\ref{lem:orth proj alpha} and 
Lemma~\ref{lem:special element F}, as at least one of the pairs $(\mu,\nu), (\mu,\kappa)$ and $(\nu,\kappa)$ satisfies the assumptions of Lemma~\ref{lem:special element F}.
\end{proof}

The remaining three results of this paper, Proposition \ref{fincl}, Proposition \ref{prop:T -> V} and Theorem \ref{mainresult} below, 
give information about  those unitaries $u\in\U(\OO_2)$ for which $\lambda_u(F)\subseteq V$, $\lambda_u(T)\subseteq V$ and 
$\lambda_u(V)\subseteq V$, respectively.

\begin{prop}\label{fincl}
Let $u \in \CU(\CO_2)$ be such that $\lambda_u(F) \subseteq V$. Then $u \in V$ if and only if $\lambda_u(\CD_2) \subseteq \CD_2$ and 
$\lambda_u(S_1S_2^*) \in \CV_2$.
\end{prop}
\begin{proof}
If $u \in V$ then clearly $\lambda_u(\CV_2) \subseteq \CV_2$, which shows that the ``only if"-direction is trivial. For the converse, assume that 
$\lambda_u(\CD_2) \subseteq \CD_2$ and $\lambda_u(S_1S_2^*) \in \CV_2$. We first show that $\lambda_u(\CV_2) \subseteq \CV_2$. For this, it is enough to show that $\lambda_u(S_\mu S_\nu^*) \in \CV_2$ for all non-empty multi-indices $\mu, \nu \in W_2$.

Assume first that $\mu, \nu \in W_2$ are non-empty multi-indices with the property that there exists some partial isometry $v \in \CV_2$ such that $w = S_\mu S_\nu^* + v \in F$. This is exactly the case if one of the following three cases is satisfied:
\begin{itemize}
\item[(i)] $(\mu,\nu) = (1(k),1(\ell))$ for some $k,\ell \geq 1$;
\item[(ii)] $(\mu,\nu) = (2(k),2(\ell))$ for some $k,\ell \geq 1$;
\item[(iii)] $1(k) \neq \mu \neq 2(k)$ and $1(\ell) \neq \nu \neq 2(\ell)$ for all $k,\ell \geq 1$.
\end{itemize}
In either of these cases,
\[
\lambda_u(S_\mu S_\nu^*) = \lambda_u(P_\mu w) = 
\lambda_u(P_\mu)\lambda_u(w) \in \CP(\CD_2) \cdot V \subseteq \CV_2.
\]

Let us now check the cases where neither of the conditions (i)-(iii) are satisfied. By assumption,
\[
\lambda_u(S_{21(k-1)}S_{1(k)}^*) = \lambda_u(S_2S_1^*)\lambda_u(P_{1(k)}) \in \CV_2 \cdot \CP(\CD_2) \subseteq  \CV_2
\]
for every $k \geq 1$. Using that $\CV_2$ is $*$-invariant, it follows from a similar argument that $\lambda_u(S_{12(k-1)} S_{2(k)}^*) \in \CV_2$ for every $k \geq 1$. Now let $\mu \in W_2$ be a non-empty multi-index such that $1(\ell) \neq \mu \neq 2(\ell)$ for all $\ell \geq 1$. Then $(\mu,12(k))$ satisfies (iii) for all $k \geq 1$ and we obtain that
\[
\lambda_u(S_\mu S_{2(k)}^*) = \lambda_u(S_\mu S_{12(k)}^*)\lambda_u(S_{12(k)}S_{2(k+1)}^*)\lambda_u(S_{2(k+1)}S_{2(k)}^*) \in \CV_2.
\]
Similarly, $\lambda_u(S_\mu S_{1(k)}^*) \in \CV_2$ for all $k \geq 1$. Furthermore, this in turn shows that
\[
\lambda_u(S_{2(\ell)}S_{1(k)}^*) = \lambda_u(S_{2(\ell)} S_{12}^*) \lambda_u(S_{12} S_{1(k)}^*) \in \CV_2
\]
for all $k,\ell \geq 1$. Consequently, $\lambda_u(S_{1(k)}S_{2(\ell)}^*) \in \CV_2$ for all $k,\ell \geq 1$ as well. This covers all cases where neither of the conditions (i)-(iii) are satisfied and $\lambda_u(\CV_2) \subseteq \CV_2$ follows.

This now implies that for $i = 1,2$,
\[
\lambda_u(S_i) = \sum_{j=1}^2 \lambda_u(S_i)\lambda_u(S_j S_j^*) = \sum_{j=1}^2 \lambda_u(S_iS_jS_j^*) \in \CV_2.
\]
Thus,
\[
u = \sum_{i = 1}^2 \lambda_u(S_i)S_i^* \in \CV_2 \cap \CU(\CO_2) = V.
\]
This concludes the proof.
\end{proof}

Replacing the Thompson group $F$ with $T$ in Proposition \ref{fincl} above leads to the following simplified condition. 

\begin{prop} \label{prop:T -> V}
Let $u \in \CU(\CO_2)$ be such that $\lambda_u(T) \subseteq V$. Then $u \in V$ if and only if $\lambda_u(\CD_2) \subseteq \CD_2$.
\end{prop}
\begin{proof}
We only have to prove the ``if"-direction. Let $u \in \CU(\CO_2)$ be such that $\lambda_u(T) \subseteq V$ and $\lambda_u(\CD_2) \subseteq \CD_2$. Let $i \in \{1,2\}$. For $j\in\{1,2\}$, there clearly exists a partial isometry $v_j \in \CV_2$ such that $w_j = S_iS_jS_j^* + v_j \in T$. By assumption, it holds for $i=1,2$ that
\[
\lambda_u(S_i) = \sum_{j=1}^2 \lambda_u(S_i)\lambda_u(S_j S_j^*) = \sum_{j=1}^2 \lambda_u(S_iS_jS_j^*) = \sum_{j = 1}^2 \lambda_u(w_j)\lambda_u(P_j) \in \CV_2.
\]
Hence we conclude that
\[
u = \sum_{i=1}^2 \lambda_u(S_i)S_i^* \in \CV_2 \cap \CU(\CO_2) =  V,
\]
which finishes the proof.
\end{proof}

Now, we are ready to give the following interesting result. 

\begin{theorem}\label{mainresult} 
Let $u \in \CU(\CO_2)$. Then $\lambda_u(V) \subseteq V$ if and only if $u \in V$.
\end{theorem}
\begin{proof}
This is an immediate corollary to Lemma \ref{prop:V->V implies D_2 -> D_2} and Proposition \ref{prop:T -> V}. 
\end{proof}

\bigskip\noindent \\
Sel\c{c}uk Barlak\\
Department of Mathematics and Computer Science \\
The University of Southern Denmark \\
Campusvej 55, DK--5230 Odense M, Denmark \\
E-mail: barlak@imada.sdu.dk \\

\medskip\noindent
Jeong Hee Hong \\
Department of Data Information \\
Korea Maritime and Ocean University \\
Busan 49112, South Korea \\
E-mail: hongjh@kmou.ac.kr \\

\medskip\noindent
Wojciech Szyma{\'n}ski\\
Department of Mathematics and Computer Science \\
The University of Southern Denmark \\
Campusvej 55, DK--5230 Odense M, Denmark \\
E-mail: szymanski@imada.sdu.dk

\end{document}